\documentclass[10pt]{article}

\usepackage[margin=1.5in]{geometry}
\usepackage{amsmath}
\usepackage{amssymb}
\usepackage{amsthm}
\usepackage[colorlinks=true,
citecolor=blue,
linkcolor=red,
urlcolor=black,
anchorcolor=black]{hyperref}
\usepackage{xcolor}
\usepackage[english]{babel}
\usepackage{graphicx}
\usepackage{longtable}

\newtheorem{theorem}{Theorem}[section]
\newtheorem{prop}[theorem]{Proposition}
\newtheorem{lemma}[theorem]{Lemma}
\newtheorem{cor}[theorem]{Corollary}
\newtheorem{defn}[theorem]{Definition}

\theoremstyle{definition}
\newtheorem{remark}[theorem]{Remark}

\newcommand{\SO} {\ensuremath {{\rm SO}}}
\newcommand{\GL} {\ensuremath {{\rm GL}}}

\newcommand{\Or} {\ensuremath {{\rm O}}}

\newcommand{\spin} {\ensuremath {{\rm spin}}}

\title{Hyperbolic manifolds without $\spin^\mathbb{C}$ structures and non-vanishing higher order Stiefel-Whitney classes}
\author{Alan W. Reid and Connor Sell}
\date{ }
\begin{document}
	
	\maketitle
	
	\begin{abstract}
	We show that in every commensurability class of cusped arithmetic hyperbolic manifolds of simplest type of dimension $2n+2\geq 6$ there are manifolds $M$ such that the Stiefel-Whitney classes $w_{2j}(M)$ 
	are non-vanishing for all $0 \leq 2j \leq n$. We also show that for the same commensurability classes there are manifolds (different from the previous ones) that do not admit a $\spin^\mathbb{C}$ structure.
\end{abstract}

	\section{Introduction}
	\label{intro}
	The Stiefel-Whitney classes $w_k(M)$ of a manifold $M$ are defined as the Stiefel-Whitney classes of the tangent bundle of $M$, denoted by $TM$ (we refer the reader to \S \ref{SW_classes} for a definition of Stiefel-Whitney classes). The question as to whether Stiefel-Whitney classes
	vanish or not has a rich history in geometric and algebraic topology. For example, $w_1(M)=0$ if and only if $M$ is orientable and $w_2(M)=0$ if and only if $M$ admits a spin structure. In addition, the non-vanishing of Stiefel-Whitney classes can be used as an obstruction to embedding/immersing a manifold into certain Euclidean spaces. For work on vanishing and non-vanishing of Stiefel-Whitney classes we refer the
	reader to \cite{Mas} and the more recent paper \cite{HKK} (which proves vanishing of Stiefel-Whitney classes for so-called moment angled manifolds). 
	
	Of most relevance to the topics of this paper is the work of Sullivan \cite{Su} which showed that every finite volume hyperbolic manifold in any dimension has a finite sheeted cover that is stably parallelizable, and hence by properties of Stiefel-Whitney classes has vanishing Stiefel-Whitney classes for all $k\geq 1$ (see \S \ref{SW_classes}). On the other hand, in \cite{LRspin}, examples of cusped orientable finite volume hyperbolic manifolds $M$ were constructed in all dimensions $\geq 5$ for which $w_2(M)\neq0$, and more recently, in \cite{MRS} Martelli, Riolo, and Slavich constructed closed orientable hyperbolic n-manifolds $M$ in all dimensions $n\geq 4$ for which $w_2(M)\neq 0$. 
	
Another natural extension of a spin structure on a manifold is that of a $\spin^\mathbb{C}$ structure on a manifold (see \S \ref{spinC} for a definition).  A spin structure naturally defines a $\spin^\mathbb{C}$ structure (see \S \ref{spinC}), and so
since all compact orientable manifolds in dimensions $2$ and $3$ admit a spin structure, they admit a $\spin^\mathbb{C}$ structure. On the other hand, there are compact orientable $4$-manifolds that do not admit a 
spin structure (for example certain flat $4$-manifolds \cite{PS}), but it is known that every orientable $4$-manifold does admit a $\spin^\mathbb{C}$ structure \cite{TV}.

In this paper we extend the results of \cite{LRspin} to construct cusped hyperbolic manifolds with non-vanishing higher order Stiefel-Whitney classes, as well as to provide examples of cusped hyperbolic manifolds that do not admit a 
$\spin^\mathbb{C}$ structure. The most precise version of our results arises in the context of arithmetic hyperbolic manifolds.
	
\begin{theorem}
\label{theorem_arithmetic}
Let $n$ be a positive integer.  For every $m \geq 2n+2\geq 6$, every cusped arithmetic hyperbolic $m$-manifold is commensurable with orientable hyperbolic $m$-manifolds $M_1$ and $M_2$ such that:
\begin{itemize}
\item $w_{2j}(M_1)\neq 0$ for all $0 \leq 2j \leq n$;
\item $M_2$ does not admit a $\spin^\mathbb{C}$ structure.
\end{itemize}
\end{theorem}
	
The proof of Theorem \ref{theorem_arithmetic} uses the methods of \cite{LRcusps} as well as the more recent ideas of \cite{Sell}, and it is perhaps helpful to pick out a key idea in the proof which allows one to quickly prove the weaker statement contained in the following theorem.
	
	\begin{theorem}
		\label{theorem_main}
		Let $n$ be a positive integer.  For every $m \geq 2n+2\geq 6$, 
		there exist infinitely many finite volume cusped orientable hyperbolic $m$-manifolds $M_1$ and $M_2$ that satisfy the conclusions of Theorem \ref{theorem_arithmetic}.
	\end{theorem}

By way of contrast, Whitney \cite{Wh} showed that if $M$ is a compact orientable $4$-manifold then $w_3(M)=0$, and Massey \cite{Mas} proved more generally that if $M$ is an even 
dimensional compact orientable $n$-manifold then $w_{n-1}(M)=0$. He also showed in \cite{Mas} that if $n=4k+3$ and $M$ is a compact orientable $n$-manifold, then $w_n(M)=w_{n-1}(M)=w_{n-2}(M)=0$.
	
The remainder of this paper is organized as follows. In \S \ref{SW_classes}, we will give some relevant background about Stiefel-Whitney classes and in \S \ref{spinC} discuss their connection to $\spin^\mathbb{C}$ structures.
In \S \ref{flat} we provide some background 
on flat manifolds, particularly those described in \cite{ImKim} and \cite{LPS}, which we then use in \S \ref{proof_main} to prove Theorem \ref{theorem_main}.  In \S \ref{nonvanish_Arith} we prove Theorem \ref{theorem_arithmetic}, which requires some background from quadratic forms and arithmetic hyperbolic manifolds. Finally in \S \ref{final} we provide other corollaries of our work.\\[\baselineskip]
\noindent{\bf Acknowledgement:}~{\em The first author would like to thank Bruno Martelli for many interesting and helpful discussions concerning the vanishing and non-vanishing of Stiefel-Whitney classes in the context of finite volume hyperbolic manifolds. Both authors are grateful to him for pointing out the examples of \cite{LPS} and suggesting our arguments could be extended to these examples.}

	\section{Stiefel-Whitney classes}
	\label{SW_classes}
	The Stiefel-Whitney classes $w_i$ are characteristic classes defined for a rank $k$ vector bundle $E$ over a manifold $M$, and are elements of the cohomology group $H^i(M, \mathbb{Z}/2\mathbb{Z})$.  A key property of the non-vanishing of Stiefel-Whitney classes is  that they can be used to obstruct sets of everywhere linearly independent sections of $E$; namely, if  
	$w_i(E) \neq 0$, then the vector bundle $E$ does not have $ k-i+1 $ everywhere linearly independent sections.  
	For convenience, we recall the definition of Stiefel-Whitney classes, which can be defined concisely in an axiomatic way as follows.
	
	\begin{defn}[Stiefel-Whitney classes {\cite[Ch. 4]{MilStash}}]
		The Stiefel-Whitney classes of a rank $k$ vector bundle $E$ over $M$ are the unique set of elements $w_i(E) \in H^i(M, \mathbb{Z}/2\mathbb{Z})$ satisfying the following axioms.
		\begin{enumerate}
			\item $w_0(E) = 1$ and $w_i(E) = 0$ for all $i > k$.
			\item For a manifold $M'$ and any continuous map $f:M^\prime\rightarrow M$, $w_i(f^\ast(E)) = f^\ast(w_i(E))$, where $f^\ast(E)$ is the pullback vector bundle.
			\item Let $E$ and $E^\prime$ be vector bundles over the base manifold $M$.  Then
			\begin{equation*}
				w_i(E \oplus E^\prime) = \sum_{j=0}^i w_j(E) \cup w_{i-j}(E^\prime)~~\hbox{(The Whitney Product Formula)}
			\end{equation*}
			where $\cup$ is the cup product.
			\item The M\"obius bundle $ E_{M\ddot{o}bius} $ over $ S^1 $ satisfies $w_1(E_{M\ddot{o}bius}) \neq 0$.
		\end{enumerate}
	\end{defn}
	
	In this paper, we will mainly be interested in the case where the vector bundle is the tangent bundle $TM$ and, as usual, we denote by $w_i(M)$ the Stiefel-Whitney classes of $TM$, and refer to these as {\em the Stiefel-Whitney classes of $M$}.

\begin{remark} 
\label{subman}
Suppose that $i: N\hookrightarrow M$ is an embedding of a submanifold of co-dimension one in the manifold $M$ with trivial normal bundle. Using the axioms of Stiefel-Whitney classes described above,  we find that $ i^\ast w_{2j}(M) = w_{2j}(N)$ (see also  \cite{Stong}). In particular this implies that if $M$ has a spin structure, then so does $N$.

More generally, this can be used to
justify a claim made in \S \ref{intro}. Recall that a $n$-manifold is called {\em stably parallelizable} if there exists a trivial vector bundle $\epsilon$ over $M$ so that $TM\oplus\epsilon$ is also a trivial bundle, in which case
as above, it follows that the Stiefel-Whitney classes $w_i(M)=0$ for $i\geq 1$ (as mentioned in \S \ref{intro}).\end{remark}

\section{$\spin^\mathbb{C}$ structures}
\label{spinC}
A $\spin^\mathbb{C}$ structure on a manifold is a complex analogue of the notion of a spin structure on a manifold and can be defined as follows \cite{Gompf}: {\em A manifold $M$ admits a $\spin^\mathbb{C}$ structure
if and only if there exists a complex line bundle $L$ such that the vector bundle $TM \oplus L$ admits a spin structure; i.e. $w_2(TM \oplus L) = 0$.}
Note that if $M$ admits a spin structure then $w_2(M)=0$, hence $w_2(TM \oplus L)=0$ and so it follows that $M$ admits a $\spin^\mathbb{C}$ structure (as noted in \S \ref{intro}).

We will need the following lemma.  It appears to be well-known, but we could not locate a proof and so include one here (c.f. Remark \ref{subman}).

\begin{lemma}
\label{restrict_spinC}
Let $M$ be a manifold that admits a $\spin^\mathbb{C}$ structure, and  $N \subset M$ a submanifold of codimension 1 with trivial normal bundle.  Then $N$ admits a $\spin^\mathbb{C}$ structure.\end{lemma}

\begin{proof} Let $\pi :TM \oplus L\rightarrow M$ be the projection map, and denote the bundle $TM \oplus L$ restricted to $N$ by $ (TM \oplus L)_N $; i.e. the set of points $\{ (p, e) \in M \times (TM \oplus L) ~|~ p \in N, \pi(e) = p \}$.
By construction, $ (TM \oplus L)_N $ is the pullback bundle induced by the inclusion map $ i: N \hookrightarrow M $, so by the properties of Stiefel-Whitney classes described in \S \ref{SW_classes} and Remark \ref{subman} we deduce:
		\begin{equation*}
			w_2((TM \oplus L)_N) = w_2(i^\ast(TM \oplus L)) = i^\ast(w_2(TM \oplus L)) = i^\ast(0) = 0.
		\end{equation*}
		Therefore, $ (TM \oplus L)_N $ admits a spin structure.
		
Similarly, let $L_N$ (resp. $TM_N$) be the complex line bundle $L$ restricted to $N$ (resp. $TM$ restricted to $N$).  Since $N$ has codimension 1 and trivial normal bundle, $TM_N = TN \oplus \epsilon$, where $\epsilon$ is
a trivial rank 1 bundle.  Using this, we can decompose $(TM \oplus L)_N$ as follows:
		\begin{equation*}
			(TM \oplus L)_N = TM_N \oplus L_N = TN \oplus \epsilon \oplus L_N = (TN \oplus L_N) \oplus \epsilon.
		\end{equation*}
		
As a trivial bundle, $\epsilon$ admits a spin structure, and we already observed that $(TM \oplus L)_N$ admits a spin structure.  By \cite[Ch. II Prop. 1.15]{LM}, given two vector bundles $A$ and $B$ over the same 
base space, if two of $A$, $B$, and $A \oplus B$ admit a spin structure, then the third does as well.   Hence, $TN \oplus L_N$ admits a spin structure, and we deduce that $N$ admits a $\spin^\mathbb{C}$ structure.\end{proof}

\section{Flat manifolds and cusp cross-sections}
\label{flat}
Recall that if $M$ is a cusped hyperbolic $m$-manifold of finite volume, the {\em cusps} of $M$  are the ends of $M$ and all such are homeomorphic to a manifold of the form $B \times \mathbb{R}^+$, where $B$ is a closed $(m-1)$-manifold that is homeomorphic to a flat manifold (i.e. the quotient of $\mathbb{R}^{m-1}/L$ where $L$ is a discrete cocompact torsion-free group of isometries of $\mathbb{R}^{m-1}$).  We refer to $B$ as a cusp cross-section (it will not be necessary for us to be explicit about the flat metric induced by the hyperbolic metric).
	Note that by Bieberbach's third theorem on crystallographic groups (see for example \cite{Ch}), there are only finitely many closed flat manifolds up to homeomorphism in each dimension, and hence there are only finitely many possible homeomorphism classes of cusp cross-sections in each dimension.
	
In the proofs of Theorem \ref{theorem_main} and Theorem \ref{theorem_arithmetic}, we make use of certain flat manifolds constructed by Im and Kim \cite{ImKim}, as well as the so-called Generalized Hantzsche-Wendt manifolds of \cite{RS}.

\subsection{The manifolds of Im and Kim}
\label{ImKimmanifolds}
The manifolds constructed by Im and Kim \cite{ImKim} have dimension $2n+1$ and satisfy $w_{2j}(N)\neq 0$ for 
$0 \leq 2j \leq n$. We briefly recall some salient features of these manifolds and their fundamental groups that we shall make use of.

The manifolds $N$ of \cite{ImKim} are described as $\mathbb{R}^{2n+1}/ \pi_n$, where $\pi_n$ is a subgroup of $\text{Isom}(\mathbb{R}^{2n+1})$ with generators $t_1, \ldots, t_{n+1}, \tau_1, \ldots, \tau_n, K$ described in terms of translations and reflections.  These are explicitly described in matrix form below using \cite[pp 270-271]{ImKim}.

 \begin{flalign*}
		t_i: v \mapsto \left[ \begin{array}{ccccccc} 1 & & & & & & \\ & \ddots & & & & 0 & \\ & & 1 & & & & \\ & & & 1 & && \\ & & & & 1 & & \\ & 0 & & & & \ddots & \\ & & & & & & 1 \end{array} \right] v + \left[ \begin{array}{c} 0 \\ \vdots \\ 0 \\ 1 \\ 0 \\ \vdots \\ 0 \end{array} \right] & &
\end{flalign*}
\begin{flalign*}
	\tau_j: v \mapsto \left[ \begin{array}{cccccccc} 1 & & & & & & & \\ & \ddots & & & & & 0 & \\ & & -1 & & & & & \\ & & & -1 & & & & \\ & & & & \ddots & & & \\ & & & & & 1 & & \\ & 0 & & & & & \ddots & \\ & & & & & & & 1 \end{array} \right] v+ \left[ \begin{array}{c} 0 \\ \vdots \\ 0 \\ 0 \\ \vdots \\ \frac{1}{2} \\ \vdots \\ 0 \end{array} \right] & &
\end{flalign*}
\begin{flalign*}
	K: v \mapsto \left[ \begin{array}{ccccccc} 1 & & & & & & \\ & \ddots & & & & 0 & \\ & & 1 & & & & \\ & & & -1 & & & \\ & & & & -1 & & \\ & 0 & & & & \ddots & \\ & & & & & & -1 \end{array} \right] v+ \left[ \begin{array}{c} \frac{1}{2} \\ \vdots \\ \frac{1}{2} \\ \frac{1}{2} \\ 0 \\ \vdots \\ 0 \end{array} \right] & &
\end{flalign*}

The $t_i$ act by unit translation along the $i$th axis.  The $\tau_j$ reflect across the coordinate hyperplanes perpendicular to the $j$th and $(j+1)$st axes, and translate by $\frac{1}{2}$ along the $(j+1+n)$th axis.  
The action of $K$ depends on whether $n$ is odd or even: when $ n $ is odd, $K$ reflects across the last $(n+1)$st coordinate hyperplanes, and when $n$ is even, across the last $n$ coordinate hyperplanes, and then in either case, translates by $\frac{1}{2}$ along each of the first $(n+1)$ axes.  We deduce from this description that all of the generators of $\pi_n$ act on $\mathbb{R}^{2n+1}$ preserving orientation, and so each $\mathbb{R}^{2n+1}/\pi_n$ is an orientable manifold. 

\begin{remark} 
\label{rk:HW} We point out that the group $\pi_1$ is the fundamental group of the so-called Hantzsche-Wendt  $3$-manifold $\rm{HW}$ for which $w_1(\rm{HW})=w_2(\rm{HW})=0$ and this is not covered by the range of Stiefel-Whitney classes provided by  \cite{ImKim}. \end{remark}	

\subsection{Generalized Hantzsche-Wendt manifolds}
\label{GenHW}
Following \cite{RS}, we call an orientable $n$-dimensional flat manifold a {\em Generalized Hantzsche-Wendt manifold} if and only if its holonomy group is an elementary abelian $2$-group of rank $n-1$. Note that a 
Generalized Hantzsche-Wendt manifold occurs as a manifold as constructed by Im and Kim in \S \ref{ImKimmanifolds} if and only the manifold is the classical Hantzsche-Wendt  $3$-manifold $\rm{HW}$ of Remark \ref{rk:HW}. 
Generalized Hantzsche-Wendt manifolds have several interesting properties, for example they are all rational homology $n$-spheres \cite{Sz}.  However, all that we will need is that their holonomy group is diagonal, 
consisting of products of matrices \cite{LPS}:

 \begin{flalign*}
		C_i =  \left[ \begin{array}{ccccccc} 1 & & & & & & \\ & \ddots & & & & 0 & \\ & & 1 & & & & \\ & & & -1 & && \\ & & & & 1 & & \\ & 0 & & & & \ddots & \\ & & & & & & 1 \end{array} \right] 
\hbox{with -1 occurring in the $ i $'th position,} \end{flalign*}

\noindent and the following proved in \cite{LPS}.

\begin{theorem}
\label{LPS_Theorem}
Let $N$ be a Generalized Hantzsche-Wendt manifold of dimension $>3$. Then $N$ does not admit a $\spin^\mathbb{C}$ structure.\end{theorem}

As noted in \cite{LPS}, Generalized Hantzsche-Wendt manifolds only exist in odd dimensions.

\section{Proof of Theorem \ref{theorem_main}}
	\label{proof_main}	
For both of the manifolds $M_1$ and $M_2$ stated in Theorem \ref{theorem_main}, the results will be proved by breaking the statement into two cases, depending on whether the dimension of the hyperbolic manifold is even or odd.
We deal with the case of constructing manifolds with non-vanishing Stiefel-Whitney classes in detail and comment on what is needed to tweak the argument to deal with $\spin^\mathbb{C}$ structures.
	
\begin{prop}
\label{evens}
Let $n\geq 2$ be a positive integer.  There exist infinitely many finite volume cusped orientable hyperbolic $(2n+2)$-manifolds $M$ such that $w_{2j}(M)\neq 0$ for all $0 \leq 2j \leq n$.
\end{prop}
	
\begin{proof} We will make use of \cite{LRcusps} and the upgrade provided by \cite{McR}, which taken together proves that every closed, flat $(m-1)$-manifold occurs as a cusp cross-section in some cusped finite volume hyperbolic $m$-manifold (indeed of infinitely many such manifolds). In particular, the manifolds $N$ of \cite{ImKim} described in \S \ref{flat} occur as the cusp cross-section of infinitely many cusped finite volume hyperbolic $(2n+2)$-manifolds $M$.
As noted in \S \ref{flat}, $N$ is orientable and \cite{LRcusps} shows that the hyperbolic manifolds $M$ can be taken to be orientable.

Since $N$ embeds in $M$ as a cusp cross-section, the normal bundle of $i: N \hookrightarrow M$ is trivial.  Since $w_{2j}(N)\neq 0$, we deduce from Remark \ref{subman} that 
$w_{2j}(M)\neq 0$ in the range stated. \end{proof}
	
\begin{prop}
\label{odds}
Let $n \geq 2$ be a positive integer.  There exist infinitely many finite volume cusped orientable hyperbolic $(2n+3)$-manifolds $M$ such that $w_{2j}(M)\neq 0$ for all $0 \leq 2j \leq n$.\end{prop}
	
\begin{proof}
Consider the closed orientable flat $(2n+1)$-manifolds $N$ described in \S \ref{flat},  and used in the proof of Proposition \ref{evens}.  Then $N \times S^1$ is a closed orientable flat $(2n+2)$-manifold for which the embedding $N\hookrightarrow N \times S^1$ has trivial normal bundle. Hence, using Remark \ref{subman}, we deduce that the Stiefel-Whitney classes $w_{2j}(N \times S^1)\neq 0$ for all $0 \leq 2j \leq n$. Again using \cite{LRcusps} and \cite{McR} we can construct infinitely many finite volume cusped orientable hyperbolic $(2n+3)$-manifolds $M$ with cusp cross-section $N \times S^1$.  As in Proposition \ref{evens} we can then deduce that $w_{2j}(M)\neq 0$ in the range stated.\end{proof}
	
\noindent The first part of Theorem \ref{theorem_main} now follows from Proposition \ref{evens} or Proposition \ref{odds} depending on $m$ being even or odd.  \qed
\begin{remark} In Theorem \ref{theorem_main}, the case of $m=5$ can be also handled, since in \cite{LRspin}, a flat $4$-manifold without a spin structure (i.e. $w_2\neq 0$) is embedded as a cusp cross-section of a cusped hyperbolic $5$-manifold.\end{remark}
	
To deal with the case of $\spin^\mathbb{C}$ structures, we first assume that $m=2n+2$ so that $m-1$ is odd. In this case, if $N$ is a Generalized Hantzsche-Wendt manifold of dimension $m-1$,  
applying \cite{LRcusps} and \cite{McR} once again shows that $N$ occurs as the cusp cross-section of infinitely many cusped orientable finite volume hyperbolic $(2n+2)$-manifolds $M$. An application of Theorem \ref{LPS_Theorem}
and Lemma \ref {restrict_spinC} now implies that such a manifold $M$ does not admit a $\spin^\mathbb{C}$ structure.

As in the case of Proposition \ref{odds}, to deal with the case $m=2n+3$, we consider the flat manifold $N \times S^1$  where in this case, $N$ is the Generalized Hantzsche-Wendt manifold given above. By Lemma \ref {restrict_spinC},
$N\times S^1$ also does not admit a $\spin^\mathbb{C}$ structure and we then argue as above.\qed

\begin{remark} Unlike the case of non-vanishing of Stiefel-Whitney classes, the case of $m=5$ cannot be handled, since as noted in \S \ref{intro}, all $4$-manifolds admit a $\spin^\mathbb{C}$ structure.\end{remark}

	\section{Commensurability classes of arithmetic manifolds}
	\label{nonvanish_Arith}
	The main result of this section is the more precise version of Theorem \ref{theorem_main} stated as Theorem \ref{theorem_arithmetic} and repeated below for convenience. The key point is that the constructions of \cite{LRcusps} and \cite{McR} exploits arithmetic manifolds which we can leverage here to prove a stronger result in the arithmetic setting.
	We will provide some background and necessary definitions in subsequent subsections.
	
	\begin{theorem}
		\label{arithmeticSW}
Let $n$ be a positive integer.  For every $m \geq 2n+2\geq 6$, every cusped arithmetic hyperbolic $m$-manifold is commensurable with orientable hyperbolic $m$-manifolds $M_1$ and $M_2$ such that:
\begin{itemize}
\item $w_{2j}(M_1)\neq 0$ for all $0 \leq 2j \leq n$;
\item $M_2$ does not admit a $\spin^\mathbb{C}$ structure.
\end{itemize}	
\end{theorem}

Note that the separability argument of \cite{LRcusps} and \cite{McR} provides, in each commensurability class of cusped arithmetic hyperbolic $m$-manifolds, infinitely many cusped orientable hyperbolic $m$-manifolds $M_1$ and $M_2$ as
in Theorem \ref{arithmeticSW}.

	\subsection{Quadratic forms}\label{quad}
	\begin{defn}[Quadratic form]
		A quadratic form over a field $K$ is a homogeneous polynomial of degree 2 with coefficients in $K$.
	\end{defn}
	For the rest of this paper, we will assume our quadratic forms are non-degenerate and defined over $ \mathbb{Q} $.  A quadratic form in $n$ variables is said to have \textit{rank} $n$.  We can write a rank $n$ quadratic form $q(x) = \sum_{i=1}^n \sum_{j=1}^n a_{ij} x_i x_j $ as a $n \times n$ symmetric matrix $Q$ such that $q(x) = x^t Q x$, by setting the entry $q_{ij}$ to be $a_{ij}$ if $i = j$ and $\frac{a_{ij}}{2}$ if not.
	\begin{defn}[Rational equivalence]
		Two rank $n$ quadratic forms given by symmetric matrices $Q_1$ and $Q_2$ are \textit{rationally equivalent} if there exists a matrix $T\in \GL(n,\mathbb{Q})$ such that $T^t Q_1T = Q_2$.
	\end{defn}
	All quadratic forms are rationally equivalent to a \textit{diagonal} form, that is, a quadratic form whose corresponding symmetric matrix is a diagonal matrix.  We will be considering quadratic forms up to rational equivalence, so we can always use a diagonal representative from each rational equivalence class.  For ease of notation, we will write the diagonal quadratic form $q(x) = \sum_{i=1}^n a_i x_i^2$ as $\langle a_1, \ldots, a_n \rangle$.  There is one more form of equivalence we need to consider.
	\begin{defn}[Projective equivalence]
		Two quadratic forms $q_1$ and $q_2$ are projectively equivalent if there are two nonzero integers $a$ and $b$ such that $aq_1$ and $bq_2$ are rationally equivalent.
	\end{defn}
	A complete set of invariants of diagonal quadratic forms $q$ up to projective equivalence is given by the signature, discriminant, and Hasse-Witt invariants $ \epsilon_p(q) $, which we will define below.  
	
Let $ q(x) = \langle a_1, \ldots, a_n \rangle $.  The signature of $q$ is defined to be $(r,s)$, where $r$ is the number of positive coefficients $a_i$ (or for non-diagonal forms, the positive eigenvalues in the corresponding symmetric matrix), and $s$ is the number of negative $a_i$.  The discriminant is defined to be the product of the coefficients modulo squares, $\prod_{i=1}^n a_i\in \mathbb{Q}^*/(\mathbb{Q}^*)^2$ 
	(or for non-diagonal forms, the determinant of the corresponding symmetric matrix).  The Hasse-Witt invariants are more complicated, and are defined using Hilbert symbols.  The Hilbert symbol is defined as follows:
	\begin{equation*}
		(a,b)_p = \Big\{ \begin{array}{ll} 1 & \text{if~} z^2 = ax^2 + by^2 \text{~has~a~solution~in~} \mathbb{Q}_p \\ -1 & \text{otherwise.} \end{array}
	\end{equation*}
	\begin{defn}[Hasse-Witt invariants] Given a diagonal quadratic form $q = \langle a_1, \ldots, a_n \rangle$ over $\mathbb{Q}$ and a prime $p$, possibly $ \infty $, the \textit{Hasse-Witt invariant of $q$ at $p$} is given by
		\begin{equation*}
			\epsilon_p(q) = \prod_{1 \leq i < j \leq n} (a_i, a_j)_p.
		\end{equation*}
	\end{defn}
	The Hasse-Witt invariants (and the Hilbert symbols) satisfy Hilbert reciprocity, meaning that $\prod \epsilon_p(q) = 1$ for any quadratic form $q$.  Additionally, for any fixed $q$, only finitely many $\epsilon_p(q)$ are equal to $-1$, because any 
	such $p$, aside from $ 2 $ and $ \infty $, must divide at least one coefficient of $ q $. 
	
	We will make use the following lemma of Serre to construct quadratic forms.
	\begin{lemma}[{\cite[Ch. IV, Prop. 7]{Serre}}]
		\label{lemma:serre}
		Let $d$, $r$, $s$, and $ n $ be integers, and $\epsilon_p$ be $1$ or $-1$ for each prime $p$, including $\infty$.  Then there exists a rank $n$ quadratic form $q$ of discriminant $d$, signature $(r,s)$, and Hasse-Witt invariants \
		$\epsilon_p$ if and only if the following conditions are satisfied.
		\begin{enumerate}
			\item $\epsilon_p = 1$ for almost all $p$ and $\prod \epsilon_p = 1$ over all primes $p$.
			\item $\epsilon_p = 1$ if $n = 1$, or if $n = 2$ and the image of $d$ in $\mathbb{Q}_p^\ast / (\mathbb{Q}_p^\ast)^2$ is $-1$.
			\item $r, s \geq 0$ and $n = r+s$.
			\item The sign of $d$ is equal to $(-1)^s$.
			\item $\epsilon_\infty = (-1)^{\frac{s(s-1)}{2}}$.
		\end{enumerate}
	\end{lemma}

	\subsection{Orthogonal groups}
	\label{ortho}
	As in \S \ref{quad} we let $q$ be a rank $n$ quadratic form defined over $\mathbb{Q}$ with associated symmetric matrix $Q$. The orthogonal group and special orthogonal group of $q$ (or $Q$) are the groups:
	\begin{equation*}
		\Or(q) = \{A\in \GL(n,\mathbb{R}) | A^tQA=Q\}~\hbox{and}~\SO(q) = \{A \in \Or(q) | \det(A) =1\}.
	\end{equation*}
	Assuming that $q$ is not a definite quadratic form, arithmetic subgroups of $\Or(q)$ can be constructed as $\Or(q,\mathbb{Z}) = \Or(q)\cap \GL(n,\mathbb{Z})$ (and the obvious subgroup $\SO(q,\mathbb{Z})$). The assumption that $q$ is not definite is simply to to ensure that the arithmetic groups constructed are infinite.
	
	If $q_1$ and $q_2$ are rationally equivalent quadratic forms then there exists $T\in\GL(n,\mathbb{Q})$ such that $T^tQ_1T=Q_2$, in which case $T^{-1}\Or(q_1,\mathbb{Z})T$ is commensurable with $\Or(q_2,\mathbb{Z})$.
	
	Note that if $q$ is a quadratic form of signature $(m,1)$, then $q$ is equivalent over $\mathbb{R}$ to the quadratic form $j_{m+1}(x) = x_1^2 + x_2^2 + \ldots + x_{m}^2 - x_{m+1}^2$ (with associated symmetric matrix $J_{m+1}$).
	
	\subsection{Arithmetic hyperbolic manifolds}\label{arith}
	For the remainder of the paper, it will be convenient to identify hyperbolic $m$-space $\mathbb{H}^m$ with the hyperboloid model
	\begin{equation*}
		\mathbb{H}^m = \{ x \in \mathbb{R}^{m+1} | j_{m+1}(x) = -1, x_{m+1} > 0 \},
	\end{equation*} 
	where $j_{m+1}$ is the quadratic form given above.  The full group of isometries of $\mathbb{H}^m$ is given by $\text{Isom}(\mathbb{H}^m) = \Or^+(j_{m+1})$ where $\Or^+(j_{m+1})$ is the subgroup of $\Or(j_{m+1})$ preserving the upper half sheet of the hyperboloid
	$\{ x \in \mathbb{R}^{m+1} | j_{m+1}(x) = -1\}$.
	
	A finite volume cusped hyperbolic $m$-manifold $M=\mathbb{H}^m/\Gamma$ is defined to be \textit{arithmetic} if $\Gamma$ is commensurable with a group constructed from an arithmetic group $\Or(q,\mathbb{Z})$ as follows. Let $q$ be a rational quadratic form of signature $(m,1)$, and 
	$T\in \GL(m+1,\mathbb{R})$ such that $T^tQT=J_{m+1}$, then the group $T^{-1}\Or(q,\mathbb{Z})T\cap\Or^+(j_{m+1})$ has finite co-volume acting on $\mathbb{H}^m$ and $\Gamma$ is arithmetic if
	it is commensurable with some such group.
	
	Note that by \cite{MA}, two cusped arithmetic hyperbolic manifolds are commensurable  if and only if their associated quadratic forms are projectively equivalent.
	
	\subsection{Proof of Theorem \ref{arithmeticSW}}
Before commencing with the proof we make some preliminary comments to help guide the reader. The proof of Theorem \ref{arithmeticSW} makes more careful use of the construction in \cite{LRcusps}.  In particular, given some closed flat $(m-1)$-manifold $B$, \cite{LRcusps} provides an algorithm for finding a cusped arithmetic orbifold 
	with some cusp cross-section being diffeomorphic to $B$.  Using \cite{McR}, we can pass to some finite cover to arrange that $B$ be a cusp cross-section of an arithmetic hyperbolic $m$-manifold.  
	The algorithm of \cite{LRcusps} outputs arithmetic hyperbolic orbifolds associated to some quadratic form $q$ (as in \S \ref{arith}) where $q$ is built into the algorithm of \cite{LRcusps} using a description of the holonomy group of the flat manifold. This holonomy group is finite, and so preserves a positive definite quadratic form.
	We can use this information to control the commensurability class of the output of the algorithm. 
	
We now give the details for the case of non-vanishing of Stiefel-Whitney classes, the case dealing with $\spin^\mathbb{C}$-structures is then sketched following this line of argument.
	
	\begin{proof} As in Propositions \ref{evens} and \ref{odds}, we find hyperbolic manifolds with the desired properties by running the closed flat $(m-1)$-manifolds $N$ of \S \ref{flat} (or taking a product of such with $S^1$ if necessary) through the algorithm from \cite{LRcusps}.  To make the argument concise we deal with the case of $m=2n+2$, and comment briefly on the case of $2n+3$ at the end of the proof.
				
The algorithm of \cite{LRcusps} depends on a representation of the holonomy group of $N$ (which in the case at hand is $(\mathbb{Z}/2\mathbb{Z})^{n+1}$) into $\GL(2n+1, \mathbb{Z})$, and a choice of quadratic form 
		$f$ of signature $(2n+1, 0)$ which is preserved by each matrix in the image of this representation. In our setting, the holonomy group is generated by the diagonal matrices described in \S \ref{flat}, and hence any diagonal quadratic form of rank $(2n+1)$ is preserved by this representation. Hence for {\em any} positive definite diagonal quadratic form $f$ of signature $(m-1,0)$, the algorithm of \cite{LRcusps} now produces a manifold in the commensurability class described as in \S \ref{arith} associated to the group $\SO(f \oplus \langle 1, -1 \rangle, \mathbb{Z})$.
Note that the translational parts of the generators listed in \S \ref{flat} contain $1/2$, but built into \cite{LRcusps} is the ability to scale so that the resultant image of the groups $\pi_n$ do lie in $\SO(f \oplus \langle 1, -1 \rangle, \mathbb{Z})$. We caution the reader that implementing the scaling described above may change the form $ f \oplus \langle 1, -1 \rangle $, but it will still lie in the same projective equivalence class, and hence the resulting manifold will lie in the same commensurability class.

Given this discussion, to complete the proof, we need to show that {\em every quadratic form} $q$ of signature $(m,1)$ is projectively equivalent to a quadratic form $q^{\prime \prime}$ that can be written as $q^\prime \oplus \langle 1, -1 \rangle$ for some 
diagonal quadratic form $q^\prime$ with signature $(m-1, 0)$ (taking $q^\prime=f$ in the above discussion).

As stated in \S\ref{quad}, 
$q$ is projectively equivalent to any form with the same signature $ (m,1) $, discriminant, and Hasse-Witt invariants.  Let the discriminant of $ q $ be $d$, and its Hasse-Witt invariants be $ h_p = \epsilon_p(q) $.  We will use Lemma \ref{lemma:serre} to construct a quadratic form $q^\prime$ with signature $(m-1,0)$, discriminant $-d$, and Hasse-Witt invariants $\epsilon_p(q^{\prime}) = h_p(-1, -d)_p$, and this will arrange that $q$ and $ q^{\prime \prime}$ have all necessary invariants equal.

We begin by proving that given $q^\prime$ with the invariants described above, $q^{\prime \prime} = q^\prime \oplus \langle 1, -1 \rangle$ must be projectively equivalent to $q$, because they have the same invariants.  Clearly, the signature of $ q^{\prime \prime} $ is $ (m,1) $ and its discriminant is $ d = (-d)(-1) $.  We can compute the Hasse-Witt invariants from the Hilbert symbols of the coefficients.  Let $ q^\prime = \langle x_1, \ldots, x_{m-1} \rangle $, so that $ q^{\prime \prime} = \langle x_1, \ldots, x_{m-1}, 1, -1 \rangle $.  Then:
\begin{align*}
	\epsilon_p(q^{\prime \prime}) & = \left( \prod_{i < j} (x_i, x_j)_p \right) \left( \prod_i (1, x_i)_p \right) \left( \prod_i (-1, x_i)_p \right) (1,-1)_p \\
	& = ( \epsilon_p(q^\prime) )(1) \left( -1, \prod_i x_i \right)_p (1) \\
	& = h_p (-1,-d)_p (-1,-d)_p \\
	& = h_p = \epsilon_p(q)
\end{align*}
Thus $ q $ and $ q^{\prime \prime} $ share their signature, discriminant, and Hasse-Witt invariants, and must be projectively equivalent.
		
To complete the proof, it remains to show that there exists $q^\prime$ with signature $(m-1,0)$, discriminant $ -d $, and Hasse-Witt invariants $ \epsilon_p(q^{\prime}) = h_p(-1, -d)_p $.  Referring to Lemma \ref{lemma:serre}, we note that condition (1) is satisfied because $ \prod_p \epsilon_p(q) = \prod_p h_p = 1 $, and Hilbert reciprocity tells us $ \prod_p (1,-d)_p = 1 $.  Condition (2) is satisfied because $ m-1 \geq 5 $, and condition (3) can be satisfied by choosing the rank to be $ m-1 $.  We know $ d < 0 $ is the discriminant of a quadratic form of signature $ (m,1) $, so $ -d > 0 $ as required by condition (4).  Finally, condition (5) holds because $ \epsilon_\infty(q^\prime) = h_\infty (-1, -d)_\infty = 1 $; note that $ h_\infty = 1 $ because $ q $ has signature $ (m,1) $, and $ (-1,-d)_\infty = 1 $ because $ d $ is negative.  By Lemma \ref{lemma:serre}, $ q^\prime $ exists, and the rest of the argument follows.

The only objects that change when we consider the case in which $ m = 2n+3 $ are $ N $ and its holonomy group.  We can take $ N = N^\prime \times S^1 $, where $ N^\prime $ is a $ (2n+1) $-manifold described in \S\ref{flat}.  The holonomy representation of $ N^\prime $ is the same as that of $ N $, but with an extra row at the bottom and column at the right, each of which is 0 except for a 1 in the bottom right corner.  This does not change the fact that the representation is diagonal, and thus any quadratic form $ f $ of the appropriate rank can be chosen in the algorithm of \cite{LRcusps}.  The rest of the proof follows unchanged. \end{proof}
	
We now discuss the case of obstructing $\spin^\mathbb{C}$ structures.  Suppose $m=2n+2$ and that $N$ is a Generalized Hantzsche-Wendt manifold of dimension $m-1$.  Then the holonomy representation of $N$ is generated by diagonal matrices, as described in \S \ref{GenHW}. The same argument used in the proof of Theorem \ref{arithmeticSW} shows that every cusped arithmetic hyperbolic $m$-manifold is commensurable to a 
manifold $M$ with a cusp cross-section homeomorphic to $N$, and the result now follows as before using an application of Theorem \ref{LPS_Theorem} and Lemma \ref {restrict_spinC}.

To handle the case $m=2n+3$, we consider the manifold $N\times S^1$, where $N$ is a Generalized Hantzsche-Wendt manifold of dimension $m-2$. Now the holonomy representation of $N \times S^1$ will still be diagonal, 
$N\times S^1$ does not admit a $\spin^\mathbb{C}$ structure by Lemma \ref {restrict_spinC}, and the rest of the argument applies as before to complete the proof.\qed

	\section{Final remarks}\label{final}

\subsection{Non-vanishing Stiefel-Whitney classes in other settings}
In \cite{LaR}, the authors consider vanishing and non-vanishing of characteristic classes and numbers of locally symmetric spaces as well as some other classes of manifolds. In \cite[Section 6]{LaR}, they raise the question: {\em Compute the characteristic classes and/or the characteristic numbers for the remaining known examples of non-positively curved Riemannian manifolds.} 
	
	A subset of these ``remaining known examples" arises from doubling cusped hyperbolic manifolds.  More formally: a
	finite-volume cusped hyperbolic $n$-manifold is homeomorphic to the interior of a compact manifold $X$ with boundary $\partial X$ (which need not be connected), with each boundary component a flat manifold of dimension $n-1$. The manifold $X$  can be doubled along $\partial X$ to form the closed manifold $DX$ (the double of $X$).  It is well-known (see for example \cite[Section 1]{Ont}) that $DX$ admits a metric of non-positive curvature.  If we now take $X$ to come from one of the cusped hyperbolic $m$-manifolds $M_1$ constructed in Theorem \ref{theorem_main} or Theorem \ref{theorem_arithmetic}, then the flat manifold $N$ (resp. $N\times S^1$) embeds in $DX$ with trivial normal bundle, and so arguing as before, $DX$ will have non-vanishing Stiefel-Whitney classes in a certain range.  Summarizing, we have the following corollary.
	
	\begin{cor}
		\label{doubles}
		Let $n$ be a positive integer.  For every $m \geq 2n+2$, there exist infinitely many closed orientable $m$-manifolds $M$ admitting a metric of non-positive curvature with non-solvable fundamental group
		such that $w_{2j}(M)\neq 0$ for all $0 \leq 2j \leq n$.\end{cor}

\subsection{Almost complex structures} An almost complex structure on a smooth manifold $M$ is defined to be an almost complex structure on $TM$; i.e. a bundle endomorphism $J : TM \rightarrow TM$ such that 
$J^2= -\rm{id}|_{TM}$. If $M$ admits an almost complex structure, then $M$ is orientable and has even dimension.

It is known that if $M$ is a smooth manifold with an almost complex structure, then $M$ admits a $\spin^\mathbb{C}$ structure \cite{LM}. Thus as a corollary of Theorem \ref{theorem_arithmetic}
we have:

\begin{cor}
\label{not_almostC} For all $m\geq 6$ and even, every cusped arithmetic hyperbolic $m$-manifold is commensurable with an orientable hyperbolic $m$-manifold which does not admit an almost complex structure.\end{cor}

Similarly, using the cusped hyperbolic $m$-manifolds $M_2$ constructed in Theorem \ref{theorem_arithmetic} we can build analogous examples to those in Corollary \ref{doubles} without a $\spin^\mathbb{C}$ structure and so without almost complex structure, namely:

\begin{cor}
\label{doubles_notAC}
For every $m \geq 2n+2$, there exist infinitely many closed orientable $m$-manifolds $M$ admitting a metric of non-positive curvature with non-solvable fundamental group that do not admit a $\spin^\mathbb{C}$ structure. If $m$ is even, these do not admit an almost complex structure.\end{cor}

	\medskip
	
	\noindent  Department of Mathematics,\\
	Rice University,\\
	Houston, TX 77005, USA.\\
	{\tt{alan.reid@rice.edu, cds7@rice.edu}}

\end{document}